\documentclass[11pt]{amsart}
\usepackage{setspace, amsmath, amsthm, amssymb, amsfonts, amscd, epic, graphicx, ulem, dsfont}
\usepackage[T1]{fontenc}
\usepackage{multirow}
\usepackage{bbm}
\usepackage{enumerate}

\makeatletter \@namedef{subjclassname@2010}{
  \textup{2010} Mathematics Subject Classification}
\makeatother

\newtheorem{thm}{Theorem}[section]
\newtheorem{cor}[thm]{Corollary}
\newtheorem{lem}[thm]{Lemma}
\newtheorem{pro}[thm]{Proposition}

\theoremstyle{remark}
\newtheorem*{rema}{Remark}

\theoremstyle{definition}

\begin{document}

\title[Proofs of the operator monotony of the square root]{New proofs of the operator monotony of the square root and the inverse}
\author[M. H. Mortad]{Mohammed Hichem Mortad}

\dedicatory{}
\thanks{}
\date{}


\address{Department of
Mathematics, University of Oran 1, Ahmed Ben Bella, B.P. 1524, El
Menouar, Oran 31000, Algeria.\newline {\bf Mailing address}:
\newline Pr Mohammed Hichem Mortad \newline BP 7085 Seddikia Oran
\newline 31013 \newline Algeria}

\email{mhmortad@gmail.com, mortad@univ-oran.dz.}

\begin{abstract}Let $A,B\in B(H)$. We present among others a simple proof of
the widely known result stating that if $0\leq A\leq B$, then $\sqrt
A\leq \sqrt B$. The same idea is used to prove that if $0\leq A\leq
B$ and $A$ is invertible, then $B$ too is invertible and $B^{-1}\leq
A^{-1}$.
\end{abstract}

\maketitle

\section{Notations}

Throughout this paper, $H$ designates a complex Hilbert space. We
say that $A\in B(H)$ is positive, and we write $A\geq 0$, if
$<Ax,x>\geq 0$ for all $x\in H$ (this implies that $A$ is
self-adjoint). If $A,B\in B(H)$ are self-adjoint, then we write
$A\leq B$ if $B-A\geq0$. If $A\leq B$, then $AC\leq BC$ for any
positive $C\in B(H)$ which commutes with $A$ and with $B$.

Recall also that a $K\in B(H)$ is called a contraction if $\|K\|\leq
1$. This is known (cf. \cite{Berberian-Book-Hilbert-Space}) to be
equivalent to \textit{any} of the following:
\begin{itemize}
  \item $||Kx||\leq ||x||$ for all $x\in H$;
  \item $KK^*\leq I$;
  \item $K^*K\leq I$.
\end{itemize}

An important result of monotony in $B(H)$ is the so-called
\textbf{L\"{o}wner-Heinz Inequality}: \textit{If $A\geq B\geq 0$,
then $A^{\alpha}\geq B^{\alpha}$ for any $\alpha\in [0,1]$.}

In particular, \textit{if $0\leq A\leq B$, then $\sqrt A\leq \sqrt
B$.}
 Another equally important result is: \textit{If $0\leq A\leq B$ and if $A$ is invertible, then
$B$ is invertible and $B^{-1}\leq A^{-1}$.}

In this short paper, we mainly present new proofs of these two well
known results.

\section{Main Results}

The key point for proving the results are the following standard
lemmata:

\begin{lem}\label{Lemma MAIN LEMMA}(see e.g. \cite{Berberian-Book-Hilbert-Space}) Let $H$ be a complex Hilbert
space. If $A,B\in B(H)$, then
\[\forall x\in H:~\|Ax\|\leq \|Bx\|\Longleftrightarrow \exists K\in B(H) \text{ contraction}:~A=KB.\]
\end{lem}

\begin{rema}
The foregoing lemma is usually utilized to characterize hyponormal
operators.
\end{rema}

\begin{lem}\label{Cauchy-Schwarz generalized Lemma}(\textbf{Generalized Cauchy-Schwarz Inequality}) Let $H$ be a complex Hilbert
space. If $A\in B(H)$ is \textbf{positive}, then
\[|<Ax,y>|^2\leq <Ax,x><Ay,y>\]
for all $x,y\in H$.
\end{lem}

\begin{thm}(\cite{Sebesteyn-Pso-Prod-First})
Let $A,B\in B(H)$ such that $A$ is positive and $AB=B^*A$. Then
there exists a unique self-adjoint operator $S\in B(H)$ such that
\[\sqrt AB=S\sqrt A.\]
\end{thm}

A glance at the proof of the previous theorem allows us to give the
following:

\begin{lem}\label{sebesteyen improved LEMMA}
Let $A,B\in B(H)$ such that $A$ is positive and $AB=B^*A$. If $B$ is
a contraction, then there exists a unique self-adjoint
\textbf{contraction} $S\in B(H)$ such that
\[\sqrt AB=S\sqrt A.\]
\end{lem}

\begin{rema}
It is known that if $A\in B(H)$ is positive and $K\in B(H)$ is a
contraction, then (unless $KA=AK$) in general:
\[KAK^*\not\leq A.\]

For example, let $S$ be the usual shift operator on $\ell^2$ and set
$A=SS^*$. Then $A$ is positive and obeys $\|A\|\leq 1$. If we choose
$K=S^*$, then $K$ is clearly a contraction. If $KAK^*\leq A$ held,
then we would obtain
\[S^*SS^*S=I\leq SS^*,\]
which is absurd! Indeed, as we already know that $SS^*\leq I$, then
we would end up with $SS^*=I$!
\end{rema}

As a consequence of the previous lemma, we have

\begin{pro}\label{KBK* leq K Propo}Let $A\in B(H)$ be positive and let $K\in B(H)$ be a
contraction. If $AK^*=KA$, then
\[K^2A=A(K^*)^2=KAK^*\leq A.\]
\end{pro}

\begin{proof}First, observe that
\[AK^*=KA\Longrightarrow A(K^{*})^2=KAK^{*}=K^2A\Longrightarrow \sqrt A(K^{*})^2=V\sqrt A\]
for some (self-adjoint) contraction $V\in B(H)$.

Now, let $x\in H$. Then
\[<KAK^*x,x>=<A(K^{*})^2x,x>\leq <Ax,x>^{\frac{1}{2}}<A(K^{*})^2x,(K^{*})^2x>^{\frac{1}{2}}\]
where we have used Lemma \ref{Cauchy-Schwarz generalized Lemma} in
the last inequality. But,
\[<A(K^{*})^2x,(K^{*})^2x>^{\frac{1}{2}}=<\sqrt A(K^{*})^2x,\sqrt A(K^{*})^2x>^{\frac{1}{2}}=\|\sqrt A(K^{*})^2x\|=\|VAx\|.\]

Since $V$ is a contraction, $\|VAx\|\leq \|Ax\|$ for each $x$.
Finally, in view of $\|Ax\|=<\sqrt A x,x>^{\frac{1}{2}}$, we obtain:
\[<KAK^*x,x>\leq <Ax,x>,\]
and this completes the proof.
\end{proof}

\begin{thm}(cf. \cite{Kad-Ring-I})\label{theorem squrae root first}
Let $A,B\in B(H)$. If $0\leq A\leq B$, then $\sqrt A\leq \sqrt B$.
\end{thm}

\begin{proof}
Let $x\in H$. Since $A\leq B$, we easily see that:
\[\|\sqrt A x\|^2\leq \|\sqrt B x\|^2.\]

So, by Lemma \ref{Lemma MAIN LEMMA}, we know that $\sqrt A=K\sqrt B$
for some contraction $K\in B(H)$. Since $\sqrt A$ is self-adjoint,
it follows that $K\sqrt B$ too is self-adjoint, that is, $K\sqrt
B=\sqrt BK^*$.

Hence, by the generalized Cauchy-Schwarz inequality:
\[<\sqrt Ax,x>^2=<K\sqrt Bx,x>^2=<\sqrt Bx,K^*x>^2\leq <\sqrt Bx,x><\sqrt BK^*x,K^*x>.\]
By Proposition \ref{KBK* leq K Propo}, we have
\[<\sqrt BK^*x,K^*x>=<K\sqrt BK^*x,x>\leq <\sqrt Bx,x>.\]
Accordingly,
\[|<\sqrt Ax,x>|^2\leq |<\sqrt Bx,x>|^2\text{ or } <\sqrt Ax,x>\leq <\sqrt Bx,x>,\]
as required.
\end{proof}

\begin{rema}
Notice that Proposition \ref{KBK* leq K Propo} can also be
established if we use the preceding theorem.
\end{rema}

For the new proof of the next result, we only need Lemma \ref{Lemma
MAIN LEMMA}. The proof is very simple and short.

\begin{thm}\label{theorem 2}
Let $A,B\in B(H)$. If $0\leq A\leq B$ and if $A$ is invertible, then
$B$ is invertible and $B^{-1}\leq A^{-1}$.
\end{thm}

\begin{proof}
As in the proof of Theorem \ref{theorem squrae root first} combined
with Lemma \ref{Lemma MAIN LEMMA}, we know that $\sqrt A=K\sqrt B$
for some contraction $K\in B(H)$. Since $\sqrt A$ is invertible (as
$A$ is), it follows that $I=(\sqrt A)^{-1}K\sqrt B$, i.e. the
self-adjoint $\sqrt B$ is left invertible and
  so $\sqrt B$ or simply $B$ is invertible (cf. \cite{Dehimi-Mortad-II}) and
\[(\sqrt B)^{-1}=(\sqrt A)^{-1}K=K^*(\sqrt A)^{-1}\]
by the self-adjointness of both $(\sqrt B)^{-1}$ and $(\sqrt
A)^{-1}$.

Let $x\in H$. Then (since $K^*$ too is a contraction)
\[<B^{-1}x,x>=\|(\sqrt B)^{-1}x\|^2=\|K^*(\sqrt A)^{-1}x\|^2\leq \|(\sqrt A)^{-1}x\|^2=<A^{-1}x,x>,\]
as needed.
\end{proof}

As far as I am aware, Theorem \ref{theorem squrae root first} has
not a very obvious proof in the literature at an elementary level
even when $A$ commutes with $B$ (cf.
\cite{Costara-Popa-exos-func-analysis-2003-trans-from-Russian}). The
following improvement of Lemma \ref{Lemma MAIN LEMMA} (kindly
communicated to me by Professor J. Stochel) makes the proof in case
of commutativity very simple.

\begin{lem}\label{Lemma MAIN Stochel improvement LEMMA}Let $H$ be a complex Hilbert
space. If $A,B\in B(H)$ are self-adjoint and $BA\geq 0$, then
\[\forall x\in H:~\|Ax\|\leq \|Bx\|\Longleftrightarrow \exists K\in B(H) \text{ \textbf{positive} contraction}:~A=KB.\]
\end{lem}

\begin{proof}\hfill
\begin{enumerate}
  \item "$\Leftarrow$":  Let $x\in H$. Then
  \[0 \leq <KBx,Bx> = <Ax,Bx> = <BAx,x>,\]
  that is, $BA\geq 0$.
  \item "$\Rightarrow$": Since $BA \geq 0$, it follows that $BA$ is self-adjoint, i.e. $AB=BA$. As a
consequence, $\ker A$ reduces $A$ and $B$, and the restriction of
$A$ to $\ker A$ is the zero operator on $\ker A$. Hence, we can
assume that $A$ is injective. Therefore, because $\ker B \subset
\ker A =\{0\}$, we see that $B^{-1}$ is self-adjoint and densely
defined. Set $K_0=AB^{-1}$. Then $K_0$ is densely defined and

\[||K_0(Bx)|| = ||AB^{-1}Bx|| = ||Ax|| \leq||Bx||,  \forall x \in H,\]

\end{enumerate}
signifying that $K_0$ is a contraction with a unique contractive
extension $K$ to the whole $H$. Since
\[<K_0(Bx), Bx> = <Ax,Bx>=<BAx,x> \geq 0\]
for all $x \in H$, we see that $K$ is positive as well. Clearly
$KBx= K_0(Bx) = Ax$ for all $x\in H$, which completes the proof.
\end{proof}

\begin{cor}Let $A,B\in B(H)$ be positive and commuting. Then:
\begin{enumerate}
  \item $0\leq A\leq B\Rightarrow \sqrt A\leq \sqrt B$.
  \item $0\leq A\leq B\Rightarrow B^{-1}\leq A^{-1}$, whenever $A$ is invertible.
\end{enumerate}
\end{cor}

\begin{proof}Since $A$ and $B$ are positive and commuting, we obviously know that
$AB\geq0$. Mimicking the argument at the beginning of the proof of
Theorem \ref{theorem squrae root first} combined with Lemma
\ref{Lemma MAIN Stochel improvement LEMMA} give: $\sqrt A=K\sqrt B$
for some \textit{positive} contraction $K$. As above, it follows
that $K\sqrt B=\sqrt BK$.
\begin{enumerate}
  \item We clearly have:
  \[0\leq K\leq I\Longrightarrow \sqrt A=K\sqrt B\leq \sqrt B,\]
and this proves the first statement.
  \item Also, $A=\sqrt B K\sqrt BK=BK^2$. Since $A$ is invertible,
  $BK^2A^{-1}=I$, i.e. the self-adjoint $B$ is right invertible and
  so $B$ is invertible (cf. \cite{Dehimi-Mortad-II}) and $B^{-1}=K^2A^{-1}$. Therefore, as
  $K^2A^{-1}=A^{-1}K^2$, then
  \[0\leq K^2\leq I\Longrightarrow B^{-1}=A^{-1}K^2\leq A^{-1},\]
  as required.
\end{enumerate}
\end{proof}

\begin{cor}Let $A,B\in B(H)$ be positive and commuting. Then
\[0\leq A\leq B\Longrightarrow A^2\leq B^2.\]
\end{cor}

\begin{proof}Since $AB\geq 0$, we know by Lemma \ref{Lemma MAIN Stochel improvement
LEMMA} that $\sqrt A=K\sqrt B$ for some positive contraction $K\in
B(H)$ and $K\sqrt B=\sqrt B K$. Hence
\[A=K\sqrt B K\sqrt B=K^2B.\]
So for all $x\in H$:
\[\|Ax\|^2=\|K^2Bx\|^2\leq \|Bx\|^2\]
\text{ or merely }
\[<A^2x,x>=<Ax,Ax>=\|Ax\|^2\leq \|Bx\|^2=<B^2x,x>,\]
as required.
\end{proof}

\end{document}